\documentclass[10pt]{elsarticle}
\usepackage{amssymb,amsfonts,amsmath,mathrsfs,mathtools}
\usepackage{graphicx,epsfig,subfigure}
\usepackage{bm}
\usepackage[margin=1in,papersize={8.5in,11in}]{geometry}
\usepackage{times}
\usepackage{multirow}
\usepackage{color}
\usepackage[linesnumbered,lined,ruled]{algorithm2e}

\SetCommentSty{mycommfont}

\renewcommand{\top}{\text{T}}

\def\vs{\vspace{0.2cm}}

\def\P{\mathcal{P}}
\def\I{\mathcal{I}}

\DeclareMathOperator*{\GL}{\text{GL}}

\newtheorem{proposition}{\bf Proposition}[section]

\newenvironment{proof}{{\noindent \bf \em Proof:}}{\hfill$\square$}

\title{Coordinate-adaptive integration of PDEs on tensor manifolds}

\begin{document}
\begin{frontmatter}

\author[lbnl]{Alec Dektor}
\author[ucsc]{Daniele Venturi\corref{correspondingAuthor}}
\ead{venturi@ucsc.edu}

\address[lbnl]{Lawrence Berkeley National Laboratory, Berkeley (CA) 94720, USA.}
\address[ucsc]{Department of Applied Mathematics, University of California Santa Cruz,  Santa Cruz (CA) 95064, USA.}

\cortext[correspondingAuthor]{Corresponding author}

\journal{ArXiv}

\begin{abstract}

We introduce a new tensor integration method for 
time-dependent PDEs that controls the tensor rank 
of the PDE solution via time-dependent 
diffeomorphic coordinate transformations.
Such coordinate transformations are generated by minimizing 
the normal component of the PDE operator relative to the 
tensor manifold that approximates the PDE solution 
via a convex functional. 
The proposed method significantly improves 
upon and may be used in conjunction with the 
coordinate-adaptive algorithm we recently proposed 
in \cite{Coordinate_flows}, which is based 
on non-convex relaxations of the rank minimization
problem and Riemannian optimization.  
Numerical applications demonstrating the effectiveness of
the proposed coordinate-adaptive tensor integration method 
are presented and discussed for prototype Liouville and 
Fokker-Planck equations.
\end{abstract}
\end{frontmatter}

\section{Introduction}
Developing efficient numerical methods 
to solve high-dimensional partial differential equations 
(PDEs) is a central task in many areas of 
engineering, physical sciences and mathematics. 
Such PDEs are often written in the form of an abstract initial/boundary value problem 
\begin{equation}
\label{nonlinear_IBVP}
\begin{cases}
\displaystyle \frac{\partial u(\bm x,t)}{\partial t} = 
G_{\bm x}(u(\bm x,t))  , \qquad t \in [0,T], \\
u(\bm x,0) = u_0(\bm x),
\end{cases}
\end{equation} 
which governs the time evolution of a quantity of 
interest $u(\bm x,t)$ (high-dimensional field) over 
a compact domain $\Omega \subseteq \mathbb{R}^d$ 
($d \gg 1$) and has temporal 
dynamics generated by the nonlinear 
operator $G_{\bm x}$. The subscript ``$\bm x$'' in 
$G_{\bm x}$ indicates that the operator can 
explicitly depend on the variables $\bm x\in \Omega$. 
For instance, 
\begin{equation}
G_{\bm x}(u) =  \bm f(\bm x) \cdot \nabla  u  + 
\nabla\cdot ( \sigma(\bm x) \nabla u) +R(u), 
\end{equation}
where $R(u)$ is a nonlinear reaction term.  
The PDE \eqref{nonlinear_IBVP} 
may involve tens, hundreds, or thousands of 
independent variables, and arises naturally 
in a variety of applications of kinetic theory, e.g., the 
Liouville equation, the Fokker-Planck equation, 
the Bogoliubov-Born-Green-Kirkwood-Yvon (BBGKY) PDF hierarchy \cite{Montgomery1976,Baalrud2019,VenturiJCPclusures2018,cercignani1988} 
or the Lundgren-Monin-Novikov (LMN) PDF hierarchy 
\cite{Lundgren,Friedrich,Hosokawa,venturi2018numerical}. 
These equations allow us to perform uncertainty quantification 
in many different physical systems including non-neutral plasmas, 
turbulent combustion, hypersonic flows, and stochastic particle 
dynamics.

Several general-purpose algorithms, e.g., based on tensor 
networks \cite{tensor_high_dim_pde,rodgers2020step-truncation,Heyrim2017,adaptive_rank,HeyrimJCP_2014,Dektor_dyn_approx,Coordinate_flows} and physics-informed machine learning \cite{Raissi,Raissi1,GK2021}, have recently been proposed to integrate the PDE \eqref{nonlinear_IBVP}. 
Tensor networks can be seen as factorizations of entangled 
objects such as multivariate functions or operators into 
networks of simpler objects which are amenable to efficient 
representation and computation. 
The vast majority of tensor algorithms currently available 
to approximate functions, operators and PDEs on tensor manifolds 
rely on canonical polyadic (CP) decomposition 
\cite{Beylkin2002,parr_tensor,BoltzmannBGK2020,Heyrim2017}, 
Tucker tensors, or tensors corresponding to binary trees such as 
functional tensor train (FTT) \cite{Bigoni_2016,Dektor_dyn_approx,OseledetsTT}
and hierarchical Tucker tensors 
\cite{Etter,Grasedyck2018,approx_rates,h_tucker_geom}. 
A compelling reason for using binary tensor trees is that they allow 
the tensor expansion to be constructed via spectral theory 
for linear operators, in particular the hierarchical 
Schmidt decomposition \cite{adaptive_rank,Dektor_dyn_approx,Griebel2019}.

Regardless of the chosen tensor format, 
the efficiency of tensor algorithms for high-dimensional PDEs 
depends heavily on the rank of the solution and the PDE operator 
in the chosen format. 
Indeed, the computational cost of tensor approximation 
algorithms scales linearly with the dimension of the system, 
e.g., the number of independent variables $d$ in the PDE 
\eqref{nonlinear_IBVP}, and polynomially with the rank. 
To address the rank-related unfavorable scaling,
we recently developed a new tensor rank reduction algorithm 
based on coordinate transformations that can significantly 
increase the efficiency of high-dimensional tensor approximation 
algorithms \cite{Coordinate_flows}. 
Given a multivariate function or operator, the algorithm determines a coordinate transformation so that the function, the operator, or the operator applied to the function in the new coordinate system has smaller tensor rank. 
In \cite{Coordinate_flows} we considered linear 
coordinate transformations, which yield a new class 
of functions that we called {tensor ridge functions}. 
Tensor ridge functions can be written analytically as
\begin{equation}
v_{\bm s}(\bm x,t) = u_{\bm r}(\bm A\bm x,t),
\label{A}
\end{equation}
where $u_{\bm r}$ is a given FTT tensor, and $v_{\bm s}$ 
is a FTT tensor with smaller rank, i.e., $|\bm s|\leq |\bm r|$. 
To compute the unknown matrix $\bm A$ (which enables tensor 
rank reduction), we developed a Riemannian gradient 
descent algorithm on some matrix manifold (e.g., $\bm A \in \text{SL}_d(\mathbb{R})$) that minimizes the {\em non-convex} cost functional 
\begin{equation}
C(\bm A)=S\left[u_{\bm r}(\bm A\bm x,t)\right],
\label{Co}
\end{equation}
where $S$ is the Schatten 1-norm. 
The numerical results we obtained in the recent paper 
\cite{Coordinate_flows} suggest that rank reduction 
based on linear coordinate transformations 
can significantly speed up tensor computations 
for PDEs while retaining accuracy.

In this paper we propose a new rank-reduction method 
for \eqref{nonlinear_IBVP} based on coordinate transformations 
which does not rely on non-convex relaxations of 
the rank minimization problem, e.g., 
the minimization of \eqref{Co}. 
The main idea is to select an infinitesimal generator 
of the coordinate transformation at each time 
which minimizes the component of the PDE operator (expressed 
in intrinsic coordinates) that is normal to the 
tensor manifold where we approximate the PDE solution.
This generates a time-dependent nonlinear coordinate 
transformation (referred to as a coordinate flow) directly from the PDE via a convex optimization problem. 
The Euler-Lagrange equations arising from this time-dependent variational principle (TDVP) allow us to determine an optimal coordinate flow that controls/minimizes the tensor rank of the solution during time integration. 

This paper is organized as follows. 
In Section \ref{sec:integration_on_flow} we develop the theory for coordinate-adaptive tensor integration of PDEs via rank-reducing coordinate flows. To this end, we begin with general nonlinear coordinate flows (diffeomorphisms) and formulate a new convex functional that links the PDE generator to the geometry of the tensor manifold and the flow (Section \ref{sec:integration_on_mflds}). 
We then restrict our attention to linear coordinate flows, i.e., ridge tensors, and develop the corresponding coordinate-adaptive tensor integration scheme. 
In Section \ref{sec:numerics} we demonstrate the 
effectiveness of the new tensor integrators by applying 
them to prototype Liouville and Fokker-Planck equations. 
Our findings are summarized in Section \ref{sec:conclusions}.

\section{Coordinate-adaptive time integration of PDEs on tensor manifolds}
\label{sec:integration_on_flow}
We begin by representing the solution 
$u(\bm x,t)$ to the PDE \eqref{nonlinear_IBVP} 
with another function 
\begin{equation}
\label{coordinate_flow_function}
v(\bm y(\bm x,t),t) = u(\bm x,t),
\end{equation} 
defined on a time-dependent curvilinear coordinate system $\bm y(\bm x, t)$ \cite{Aris,Serrin,VenturiConvective} with $\bm y(\bm x, 0)=\bm x$. 
It is assumed that $\bm y(\bm x,t)$ is a diffeomorphism\footnote{Recall that if $\bm y(\bm x,t)$ is a diffeomorphism then there exists a unique differentiable inverse flow $\bm x(\bm y,t)$ such that $\bm x(\bm y(\bm x,t),t)=\bm x$ for all  $t \in [0,T]$  ($T<\infty$).} which we will refer to as {\em coordinate flow}.
To simplify notation we may not explicitly write 
the dependence of $\bm y$ on $\bm x$ 
(and vice-versa the dependence of $\bm x$ on $\bm y$), or 
the dependence of $\bm x$, $\bm y$ on $t$. 
However, it is always assumed that 
$\bm x$ and $\bm y$ are related 
via a (time-dependent) diffeomorphism. 
%
%
%
Next, we define two distinct 
time derivatives of $v(\bm y(\bm x,t),t)$. 
The first represents a change in $v$ at fixed 
location $\bm y$, i.e., a derivative of $v$ with 
respect to time with $\bm y$ constant, 
which we denote by $\partial v(\bm y(\bm x,t),t)/\partial t$. 
The second represents a change in $v$ along 
the coordinate flow $\bm y$, i.e., a derivative 
of $v$ with respect to time with $\bm x$ constant, which 
we denote by $D v(\bm y(\bm x,t),t)/Dt$. The derivative 
$D v(\bm y(\bm x,t),t)/Dt$ is known as material 
derivative \cite{Aris,Haoxiang}, or convective derivative \cite{VenturiConvective,Conjugateflows}. 
Of course these time derivatives are related 
via the equation 
\begin{equation}
\label{total_derivative}
\frac{D v(\bm y(\bm x,t),t)}{D t} = \frac{\partial v(\bm y(\bm x,t),t)}{\partial t} + \frac{\partial \bm y(\bm x,t)}{\partial t} \cdot \nabla_{\bm y} v(\bm y(\bm x,t),t). 
\end{equation}
Since $u(\bm x,t)$ is a solution to \eqref{nonlinear_IBVP}, 
$v$ is related to $u$ by \eqref{coordinate_flow_function} and 
$\bm y(\bm x,0) = \bm x$, it follows immediately that $v$ satisfies the PDE
\begin{equation}
\label{nonlinear_IBVP_y}
\begin{cases}
\displaystyle \frac{D v(\bm y,t)}{D t} = 
G_{\bm y}(v(\bm y,t)) , \qquad t \in [0,T], \\
v(\bm y,0) = u_0(\bm y),
\end{cases}
\end{equation} 
where the operator $G_{\bm y}$ can be derived by 
writing $G_{\bm x}$ in coordinates $\bm y$ 
using standard tools of differential geometry 
\cite{Aris,VenturiConvective,Conjugateflows}. 
Combining \eqref{total_derivative} and \eqref{nonlinear_IBVP_y} 
we obtain\footnote{The 
PDE \eqref{total_derivative_PDE} represents the evolution 
equation for the solution to \eqref{nonlinear_IBVP_y} written along the coordinate flow $\bm y(\bm x,t)$.} 
\begin{equation}
\label{total_derivative_PDE}
\begin{cases}
\displaystyle \frac{\partial v(\bm y,t)}{\partial t} = Q_{\bm y}(v(\bm y,t),\dot{\bm y}), \qquad t \in [0,T], \\
v(\bm y,0) = u_0(\bm y), 
\end{cases}
\end{equation} 
where  
\begin{equation}
Q_{\bm y}(v(\bm y,t),\dot{\bm y})=G_{\bm y}(v(\bm y,t)) - \dot{\bm y}(\bm x,t) \cdot \nabla_{\bm y} v(\bm y,t) 
\label{Qdef}
\end{equation}
and
$\dot{\bm y}(\bm x,t)=\partial \bm y(\bm x,t)/\partial t$. 
Of course $\dot{\bm y}(\bm x,t)$ can be expressed in coordinates $\bm y$ via the inverse map $\bm x(\bm y,t)$. Moreover, the PDE domain $\Omega$ is mapped into 
\begin{equation}
\Omega_y(t) =\left\{\bm z\in \mathbb{R}^d: \bm z=\bm y(\bm x,t), \quad \bm x\in \Omega\right\}.
\end{equation}
by the coordinate flow $\bm y(\bm x,t)$.

\subsection{Dynamic tensor approximation}
\label{sec:integration_on_mflds}
Denote by $\mathcal{M}_{\bm r}(t) \subseteq \mathcal{H}(\Omega_y(t))$ a tensor manifold embedded in the Hilbert space 
\begin{equation}
\mathcal{H}(\Omega_y(t))=L^2(\Omega_y(t))
\label{Hilbert}
\end{equation}
of functions defined on the time-dependent domain $\Omega_y(t)\subseteq \mathbb{R}^d$. 
Such manifolds include hierarchical tensor formats such as the hierarchical Tucker and the tensor train formats \cite{geometric_structure_tensors,approx_rates}. 
We approximate the solution $v(\bm y,t)$ to the PDE \eqref{total_derivative_PDE} with an element $v_{\bm r}(\bm y,t)$ 
belonging to $\mathcal{M}_{\bm r}(t)$. 
We allow the tensor manifold to depend on $t$ so that the solution rank can be chosen adaptively during time integration to ensure an accurate approximation. 
At each point $v_{\bm r} \in \mathcal{M}_{\bm r}$ the function
space $\mathcal{H}(\Omega_y)$ can be partitioned as a 
direct sum of two vector subspaces  \cite{adaptive_rank}
\begin{equation}
\label{tangent_normal_partition}
    \mathcal{H}(\Omega_y) = T_{v_{\bm r}} \mathcal{M}_{\bm r} \oplus N_{v_{\bm r}} \mathcal{M}_{\bm r},
\end{equation}
where $T_{v_{\bm r}} \mathcal{M}_{\bm r}$ (resp. $N_{v_{\bm r}} 
\mathcal{M}_{\bm r}$) denotes the 
tangent space (resp. normal space) relative to the 
manifold $\mathcal{M}_{\bm r}$ at the point $v_{\bm r}$. 
Assuming the solution to \eqref{total_derivative_PDE} 
at time $t$ can be represented by an element of the tensor 
manifold $\mathcal{M}_{\bm r}(t)$, we utilize the direct 
sum \eqref{tangent_normal_partition} to decompose the 
right hand side of the PDE \eqref{total_derivative_PDE} 
into a tangential component and a normal 
component relative to $\mathcal{M}_{\bm r}(t)$ 
\begin{equation}
\begin{aligned}
\label{tangent_plus_normal}
   Q_{\bm y}(v_{\bm r},\dot{\bm y}) 
   = \P_{T}({v}_{\bm r})Q_{\bm y}(v_{\bm r},\dot{\bm y}) + \P_{N}({v}_{\bm r})Q_{\bm y}(v_{\bm r},\dot{\bm y}),\qquad v_{\bm r}\in \mathcal{M}_{\bm r}(t), 
\end{aligned}
\end{equation}
where $\P_{T}({v}_{\bm r})$ and $\P_{N}({v}_{\bm r})=\I-\P_{T}({v}_{\bm r})$ denote, respectively, the orthogonal projections 
onto the tangent and normal space of $\mathcal{M}_{\bm r}(t)$ at the point $v_{\bm r}$ (see Figure \ref{fig:TDVP}).
Projecting the initial condition in \eqref{total_derivative_PDE} 
onto $\mathcal{M}_{\bm r}(0)$ with a truncation operator 
$\mathfrak{T}_{\bm r}(\cdot)$ and 
using only the tangential component of the 
PDE dynamics for all $t \in [0,T]$ 
we obtain the projected PDE 
\begin{equation}
\label{nonlinear_IBVP_proj}
\begin{cases}
\displaystyle \frac{\partial v_{\bm r}(\bm y,t)}{\partial t} = 
\P_{T}(v_{\bm r}) Q_{\bm y}(v_{\bm r}(\bm y,t),\dot{\bm y}), \qquad t \in [0,T], \\
v_{\bm r}(\bm y,0) = \mathfrak{T}_{\bm r} (u_0(\bm y)),
\end{cases}
\end{equation} 
with solution $v_{\bm r}(\bm y,t)$ which remains on 
the tensor manifold $\mathcal{M}_{\bm r}(t)$ for 
all $t \in [0,T]$. 
We have previously shown \cite{adaptive_rank} that 
the approximation error 
\begin{equation}
\label{dyn_approx_error}
\epsilon(t)= \left\|v(\bm y,t) - v_{\bm r}(\bm y,t)\right\|_{\mathcal{H}(\Omega_y(t))}
\end{equation}
can be controlled by selecting the rank $\bm r(t)$ at each $t$ 
so that the norm of the normal component in \eqref{tangent_plus_normal} 
remains bounded by some user-defined threshold $\varepsilon$, i.e., 
\begin{equation}
\label{normal_bound}
\left\| \P_{N}(v_{\bm r})Q_{\bm y}(v_{\bm r}(\bm y,t),\dot{\bm y}) \right\|_{\mathcal{H}(\Omega_y(t))} < \varepsilon. 
\end{equation}
Therefore the normal component determines the increase in solution rank required to maintain an accurate approximation during time integration. 
On the other hand, rank decrease during time integration can be 
interpreted as the tangential component of the PDE operator 
guiding the solution to a region of the manifold 
$\mathcal{M}_{\bm r}(t)$ with higher curvature. In fact, 
recall that the smallest singular value of $v_{\bm r}$ is inversely 
proportional to the curvature of the manifold 
$\mathcal{M}_{\bm r}(t)$ at the point $v_{\bm r}$ 
\cite{Lubich_2016}. 

\subsection{Variational principle for rank-reducing coordinate flows}
\label{sec:TDVP}
\begin{figure}[!t]
\centering
\includegraphics[scale=0.5]{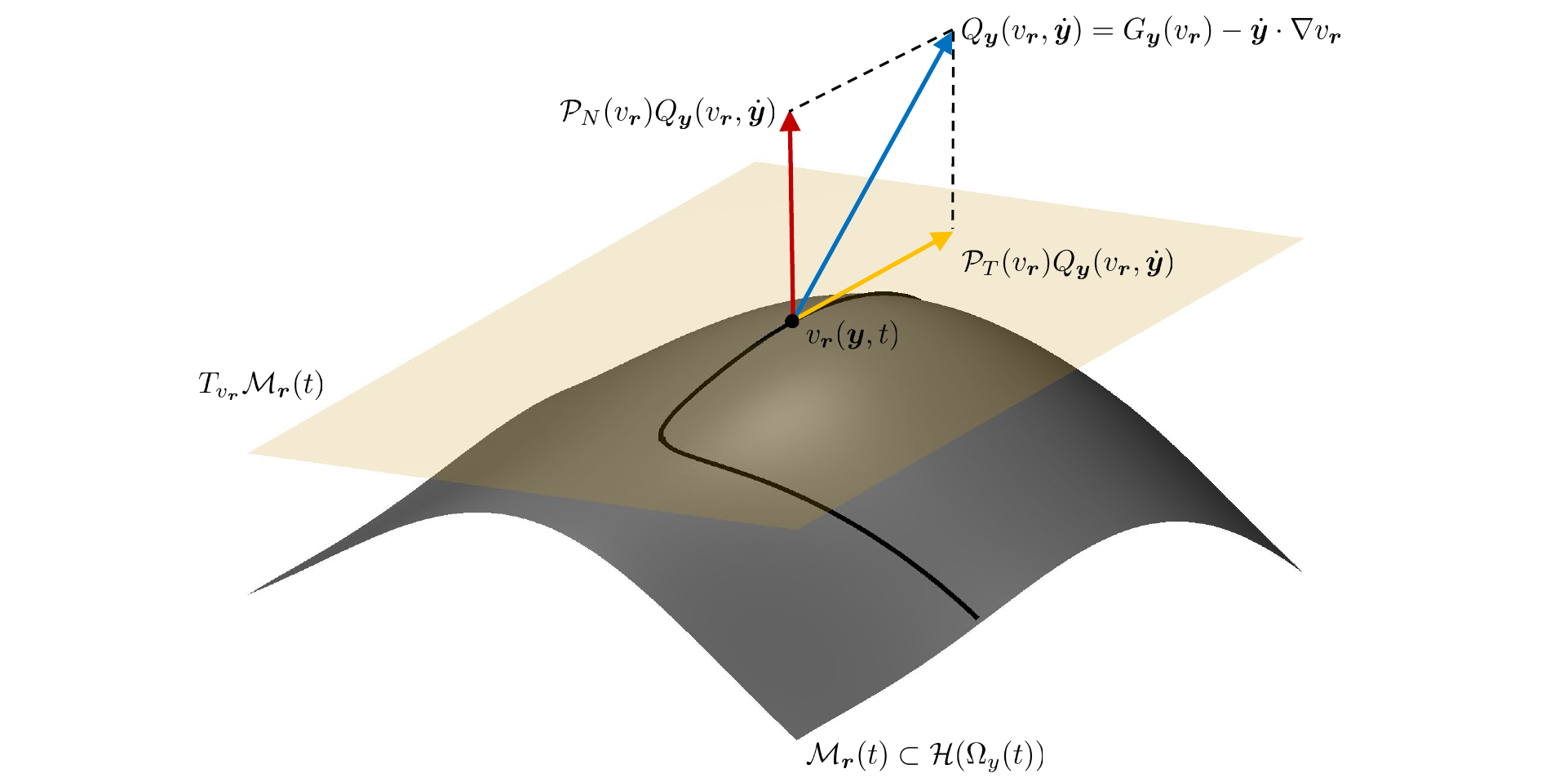} 
\caption{Sketch of the tensor manifold $\mathcal{M}_{\bm r}(t)\subset \mathcal{H}(\Omega_y(t))$. Shown are the normal and tangent components of $Q_{\bm y} (v_{\bm r},\dot {\bm y})$ (PDE operator in equation \eqref{Qdef}) at $v_{\bm r}\in \mathcal{M}_{\bm r}(t)$. The manifold $\mathcal{M}_{\bm r}(t)$ has multilinear rank that depends on the coordinate flow $\bm y(\bm x,t)$ (see \cite{Coordinate_flows}). The variational principle  \eqref{TDVP} minimizes the component of $Q_{\bm y}$ normal to the tensor manifold $\mathcal{M}_{\bm r}(t)$ at each time, i.e., the component of $Q_{\bm y}$ that is responsible for the rank increase, relative to arbitrary variations of the coordinate flow generator. This mitigates the tensor rank increase when integrating \eqref{nonlinear_IBVP_proj} using the rank-adaptive tensor methods we recently developed in \cite{adaptive_rank,rodgers2020step-truncation}.} 
\label{fig:TDVP}
\end{figure}
Leveraging the geometric structure of the tensor manifold 
$\mathcal{M}_{\bm r}(t)$ described in Section 
\ref{sec:integration_on_mflds}, we aim at generating 
a coordinate flow $\bm y(\bm x,t)$  
that minimizes the rank of the PDE solution 
$v_{\bm r}(\bm y,t)$ to \eqref{nonlinear_IBVP_proj} 
while maintaining an accurate approximation to 
\eqref{nonlinear_IBVP_y}. 
Since the approximation error \eqref{dyn_approx_error} 
is controlled by the normal component of the PDE 
operator we select an infinitesimal generator 
for the coordinate flow $\dot{\bm y}$ that minimizes 
such normal component, 
i.e., that minimizes the left hand side of \eqref{normal_bound}. 
%
%
%
%
%
%
This idea results in a 
{\em convex optimization problem} over the 
Lie algebra $\mathfrak{g}$ of infinitesimal coordinate flow
generators $\dot{\bm y}$ at time $t$ 
\begin{align}
\min_{\dot{\bm y} \in \mathfrak{g}}\left\| \P_{N}(v_{\bm r})Q_{\bm y}(v_{\bm r},\dot{\bm y}) \right\|_{\mathcal{H}(\Omega_y(t))}^2
= 
\min_{\dot{\bm y} \in \mathfrak{g}}\left\| \P_{N}(v_{\bm r})\left[ G_{\bm y}(v_{\bm r}) - \dot{\bm y}(\bm x,t) \cdot \nabla_{\bm y} v_{\bm r}  \right] \right\|_{\mathcal{H}(\Omega_y(t))}^2 .
\label{TDVP}
\end{align}

\noindent
Note that this functional is not optimizing the projection of $\dot{\bm y}\cdot \nabla v_{\bm r}$ onto the tangent space of the tensor manifold at $v_{\bm r}$. In principle, it is possible to control such projection, and eventually have the tangential component of $Q_{\bm y}$ pointing towards a region of the tensor manifold with higher curvature. This results in smaller singular values of the tensor decomposition of the solution as time increases\footnote{Recall that the curvature of a tensor manifold is inversely proportional to the inverse of the singular values \cite{Lubich_2016}.}, which may induce 
further tensor rank reduction via thresholded tensor truncation.

%
%
%
%
%

\begin{proposition}
Let $\bm f$ be a solution of the convex optimization 
problem \eqref{TDVP}. Then $\bm f$ satisfies the linear system of equations 
\begin{equation}
\label{linear_system_for_f}
 \P_N(v_{\bm r}) \left[ \bm f \cdot \nabla v_{\bm r} \right] \nabla v_{\bm r} = \P_N(v_{\bm r})\left[ G_{\bm y}(v_{\bm r}) \right] \nabla v_{\bm r} . 
\end{equation}
\end{proposition}
\begin{proof}
Due to the convexity of \eqref{TDVP} 
any critical point of the cost functional 
\begin{align}
C(\dot{\bm y})= \left\| \P_{N}(v_{\bm r})\left[ G_{\bm y}(v_{\bm r}) - \dot{\bm y}(\bm x,t) \cdot \nabla_{\bm y} v_{\bm r} \right] \right\|_{\mathcal{H}(\Omega_y(t))}^2
\label{TDVP_cost_fun}
\end{align}
in the Lie algebra $\mathfrak{g}$ is necessarily a global 
solution to the optimization problem \eqref{TDVP}. 
The first variation of $C$ with respect to 
$\dot{y}_i$ is easily obtained as 
\begin{equation}
\begin{aligned}
\delta_{\dot{y}_i} C \eta_i 
&= 2 \int_{\Omega_y(t)} \P_{N}(v_{\bm r}) \left[ G_{\bm y}(v_{\bm r}) - \dot{\bm y} \cdot \nabla v_{\bm r} \right]  \P_{N}(v_{\bm r}) \left[ - \frac{\partial v_{\bm r}}{\partial y_i} \eta_i  \right] d\bm y \\
&= 2 \int_{\Omega_y(t)} \P_{N}(v_{\bm r}) \left[ G_{\bm y}(v_{\bm r}) - \dot{\bm y} \cdot \nabla v_{\bm r} \right]
\left( - \frac{\partial v_{\bm r}}{\partial y_i} \eta_i \right) d\bm y, 
\end{aligned}
\label{TTv}
\end{equation}
where $\eta_i \in \mathfrak{g}$ is an arbitrary 
perturbation. 
To obtain the second equality in \eqref{TTv} we used the fact 
that the orthogonal projection 
$\P_N(v_{\bm r})$ is symmetric with respect to 
the inner product in $L^2(\Omega_y(t))$ and idempotent. 
Setting \eqref{TTv} equal to zero for arbitrary $\eta_i$ yields the Euler-Lagrange equations 
\begin{equation}
\label{E-L-scalar}
\P_{N}(v_{\bm r}) \left[ \dot{\bm y} \cdot \nabla v_{\bm r} \right] \frac{\partial v_{\bm r}}{\partial y_i} = 
\P_{N}(v_{\bm r}) \left[ G_{\bm y}(v_{\bm r})\right]  \frac{\partial v_{\bm r}}{\partial y_i} 
, \qquad i = 1,2,\ldots, d. 
\end{equation}
Writing the preceding system of equations in vector notation 
proves the Proposition. 

\end{proof}

\vs
\noindent
{
Note that if $\nabla v_{\bm r}$ is non-zero then 
from the optimality conditions we obtain 
$\P_N(v_{\bm r}) \left[ \bm f \cdot \nabla v_{\bm r}\right]   = \P_N(v_{\bm r}) G_{\bm y}(v_{\bm r})$. 
In this case the normal component of $v_{\bm r}(\bm y,t)$, i.e., the right hand side of \eqref{normal_bound}, is zero along the 
optimal coordinate flow and therefore the rank of $v_{\bm r}$ never increases during time integration. 
}

With the infinitesimal generator $\bm f(\bm y,t)$ 
solving the convex optimization problem \eqref{TDVP} 
available (i.e., the generator satisfying the linear 
system \eqref{linear_system_for_f}), we can 
now compute a coordinate transformation $\bm y(\bm x,t)$ via 
the dynamical system 
%
%
\begin{equation}
\label{coordinate_flow_DE}
\begin{cases}
\displaystyle \frac{\partial \bm y}{\partial t} = 
\bm f(\bm y,t) , \\
\bm y(\bm x,0) = \bm x. 
\end{cases}
\end{equation}
Such coordinate transformation allows us to 
solve the PDE \eqref{nonlinear_IBVP_proj} along 
a coordinate flow that minimizes the projection 
of $Q_{\bm y}$ onto the normal component of the tensor 
manifold in which the solution lives, henceforth 
minimizing the rank increase of $v_{\bm r}$ during time 
integration. 
The corresponding evolution equation 
along the rank reducing coordinate flow is obtained 
by coupling the ODE \eqref{coordinate_flow_DE} with 
\eqref{nonlinear_IBVP_proj} as
\begin{equation}
\label{coordinate_flow_PDE2}
\begin{cases}
\displaystyle\frac{\partial v_{\bm r}(\bm y,t)}{\partial t} = \P_T(v_{\bm r})\left[ G_{\bm y}(v_{\bm r}(\bm y,t)) - \bm f(\bm y,t) \cdot \nabla_{\bm y} v_{\bm r}(\bm y,t) \right], \\
v_{\bm r}(\bm y(\bm x,0),0) = \mathfrak{T}_{\bm r}(u_0(\bm y)). 
\end{cases}
\end{equation}
The coupled system of equations \eqref{TDVP}, 
\eqref{coordinate_flow_DE}, and \eqref{coordinate_flow_PDE2}
allows us to devise a time integration 
scheme for the PDE \eqref{nonlinear_IBVP} that leverages diffeomorphic coordinate flows to control the component of $Q_{\bm y}$ that is normal to tensor manifold $\mathcal{M}_{\bm r}$, and therefore control the rank increase of the PDE solution $v_{\bm r}$ in time.

To describe the scheme in more detail 
let us discretize the temporal domain $[0,T]$ into 
$N+1$ evenly spaced time instants 
\begin{equation}
\label{time_discretization}
t_k = k \Delta t, \qquad \Delta t = \frac{T}{N}, \qquad k = 0,1,\ldots,N. 
\end{equation}
To integrate the solution and the coordinate flow 
from time $t_k$ to time $t_{k+1}$ we first 
solve the linear system of equations \eqref{linear_system_for_f}. This gives us an optimal infinitesimal coordinate flow 
generator $\bm f(\bm y,t_k)$ at time $t_k$. 
Using the evolution equations 
\eqref{coordinate_flow_DE} and \eqref{coordinate_flow_PDE2} 
we then integrate both the coordinate flow $\bm y(\bm x,t)$ and the 
PDE solution along the coordinate flow 
$v_{\bm r}(\bm y(\bm x,t),t)$ from time 
$t_k$ to time $t_{k+1}$. 
Finally we update the PDE operator $G_{\bm y}$ 
to operate in the new coordinate system.

The proposed coordinate flow time marching scheme 
comes with two computational challenges. 
The first challenge is computing the optimal infinitesimal 
generator $\bm f(\bm y,t_k)$, i.e., solving the 
linear system \eqref{linear_system_for_f} at each time step. 
If we represent $\bm f(\bm y,t_k)$ in the span of a 
finite-dimensional basis then  
\eqref{linear_system_for_f}
yields a linear system for the coefficients 
of the expansion.
The size of such linear system 
depends on the dimension of the chosen
space. For instance,  in the case of 
linear coordinate flows we obtain a 
$d^2\times d^2$ system.
The second challenge is 
representing the PDE operator in 
coordinates $\bm y(\bm x,t)$, 
i.e., constructing $G_{\bm y}$. 
For a general nonlinear coordinate transformation 
$\bm y(\bm x,t)$ the operator $G_{\bm y}$ includes 
the metric tensor of the transformation which can 
significantly complicate the form of $G_{\bm y}$ 
(see, e.g., \cite{Aris,Haoxiang}). 
Hereafter we narrow our attention to 
linear coordinate flows and derive the corresponding 
time integration scheme which 
does not suffer from the aforementioned 
computational challenges.

\subsubsection{Linear coordinate flows}

Any coordinate flow $\bm y(\bm x,t)$ 
that is linear in $\bm x$ 
at each time $t$ can be represented by 
a time-dependent invertible 
$d \times d$ matrix $\bm \Gamma(t)$ 
\begin{equation}
\label{linear_coordinate_flow}
\bm y(\bm x,t) = \bm \Gamma(t) \bm x. 
\end{equation}
Here we consider arbitrary matrices $\bm \Gamma(t)$ belonging 
to the collection of all\footnote{It is also possible to generate linear coordinate flows using subgroups of 
$\GL_d(\mathbb{R})$ such as the subgroup of matrices with 
determinant equal to 1 (i.e. volume preserving linear maps), 
or the group of rigid rotations in $d$ dimensions.} $d \times d$ invertible matrices with real entries, denoted by $\GL_d(\mathbb{R})$. 
The corresponding collection 
of infinitesimal coordinate flow generators 
is the vector space of all $d \times d$ matrices 
with real entries.
%
In this setting the time-dependent tensor approximation 
of $u(\bm x,t)$ in coordinates $\bm y$ takes the form 
\begin{equation}
\label{ridge_time_dependent}
v_{\bm r}(\bm \Gamma(t) \bm x, t) \approx u(\bm x,t).
\end{equation}
This is known as a (time-dependent) generalized {\em tensor ridge function} \cite{Coordinate_flows,Pinkus_ridge}.   
The time-dependent optimization problem \eqref{TDVP} 
that generates the optimal linear coordinate flow can be now
written as  
\begin{equation}
\label{TDVP_linear}
\min_{\dot{\bm \Gamma} \in M_{d\times d}(\mathbb{R}) }
C\left(\dot{\bm \Gamma} \bm y\right),
\end{equation}
where $C(\cdot)$ is the cost function defined in 
\eqref{TDVP_cost_fun}. 
\begin{proposition}
\label{prop:TDVP_linear}
Let  $\bm \Sigma(t)$ be a solution of the optimization problem \eqref{TDVP_linear} at time $t$. Then $\bm \Sigma(t)$ 
satisfies the symmetric $d^2 \times d^2$ linear system 
\begin{equation}
\label{TDVP_linear_system}
\bm A {\rm vec}(\bm \Sigma(t)) = \bm b ,
\end{equation}
where ${\rm vec}(\bm \Sigma(t))$ is a vectorization 
of the $d \times d$ matrix $\bm \Sigma(t)$, 
\begin{equation}
\bm A = \int_{\Omega_y(t)} \P_{N}(v_{\bm r}) \left[ {\rm vec}(\bm c(\bm y))\right] {\rm vec}(\bm c(\bm y))^{\top} d \bm y, \qquad 
\bm b = \int_{\Omega_y(t)}  \P_{N}(v_{\bm r})\left[ G_{\bm y}(v_{\bm r}(\bm y,t))\right] 
{\rm vec}(\bm c(\bm y)) d \bm y, 
\end{equation}
and $\bm c(\bm y)$ is the $d \times d$ matrix with entries 
\begin{equation}
c_{ij}(\bm y) = \frac{\partial v_{\bm r}(\bm y,t)}{\partial y_i} y_j, \qquad i,j=1,2,\ldots,d. 
\end{equation}
\end{proposition}

\begin{proof}
First note that the cost function $C\left(\dot{\bm \Gamma} \bm y\right)$ is convex in $\dot{\bm \Gamma}$ and the 
search space $M_{d\times d}(\mathbb{R})$ is linear. 
Hence, any critical point of the cost 
function is necessarily a global minimum. 
To find such a critical point we 
calculate the derivative of $C(\dot{\bm \Gamma}\bm y)$ with 
respect to each entry of the matrix $\dot{\bm \Gamma}$ 
directly 
\begin{equation}
\begin{aligned}
\frac{\partial C(\dot{\bm \Gamma} \bm y)}{\partial \dot{\Gamma}_{ij}} 
&= 2 \int_{\Omega_y(t)} 
\P_{N}(v_{\bm r}) \left[ G_{\bm \Gamma}(v_{\bm r})  - (\dot{\bm \Gamma} \bm y) \cdot \nabla v_{\bm r}\right]
\P_{N}(v_{\bm r}) \left[ -y_j \frac{\partial v_{\bm r}}{\partial y_i} \right]
d \bm y \\
&= 2 \int_{\Omega_y(t)} 
\P_{N}(v_{\bm r}) \left[ G_{\bm \Gamma}(v_{\bm r})  - (\dot{\bm \Gamma} \bm y) \cdot \nabla v_{\bm r} \right] \left( -y_j \frac{\partial v_{\bm r}}{\partial y_i} \right)
d \bm y,
\end{aligned}
\end{equation}
for all $i,j=1,2,\ldots,d$, 
where in the second line we used the fact 
that the orthogonal projection 
$\P_N(v_{\bm r})$ is symmetric with respect to 
the inner product in $L^2(\Omega_y(t))$ and idempotent. 
The critical points are then obtained by 
setting $\partial C(\dot{\bm \Gamma})/\partial \dot{\Gamma}_{ij} = 0$ 
for all $i,j = 1,2,\ldots,d$, 
resulting in a linear system of 
equations for the entries of the infinitesimal 
linear coordinate flow generator $\dot{\bm \Gamma}$ 
\begin{equation}
\label{TDVP_linear_system_scalar}
\sum_{k,p=1}^d \dot{\Gamma}_{kp} 
\int_{\Omega_y(t)} \P_{N}(v_{\bm r}) \left[ y_p \frac{\partial v_{\bm r}}{\partial y_k} \right]
y_j \frac{\partial v_{\bm r}}{\partial y_i}
d \bm y = 
\int_{\Omega_y(t)}
\P_{N}(v_{\bm r})\left[ G_{\bm y}(v_{\bm r}) \right] y_j \frac{\partial v_{\bm r}}{\partial y_i} d \bm y, 
\end{equation}
for $i = 1,2,\ldots, d$. 
We obtain \eqref{TDVP_linear_system} 
by writing the linear system 
\eqref{TDVP_linear_system_scalar} 
in matrix notation. 

\end{proof}

\vs
\noindent
The linear system of equations 
\eqref{TDVP_linear_system_scalar} can also be obtained 
directly from the Euler Lagrange equations \eqref{E-L-scalar} 
by substituting $\dot{\bm y} = \dot{\bm \Gamma} \bm y$ and 
projecting onto $y_j$ for $j = 1,2,\ldots,d$. 

A few remarks are in order regarding the linear system 
\eqref{TDVP_linear_system}. 
The matrix $\bm A$ is a $d^2 \times d^2$ 
symmetric matrix which is determined by $(d^4 + d^2)/2$ 
entries and the vector $\bm b$ has length $d^2$. 
Computing each entry of the matrix $\bm A$ and the vector 
$\bm b$ requires evaluating a $d$-dimensional integral. 
Therefore computing $\bm A$ and $\bm b$ to set up the linear 
system at each time $t$ requires evaluating a total of 
$(d^4 + 3d^2)/2$ $d$-dimensional integrals. 
Since $v_{\bm r}$ is a low-rank tensor these 
high-dimensional integrals can computed by applying 
one-dimensional quadrature rules to the tensor modes. 
By orthogonalizing the low-rank tensor expansion (e.g., tensor 
train) it is possible to reduce the number one-dimensional 
integrals needed to compute $\bm A$. 
If the matrix $\bm A$ is singular then the solution to 
the linear system of equations is not unique. 
In this case any solution will suffice for generating a coordinate flow. 

{ 
For linear coordinate transformations it may be advantageous to consider the related functional 
\begin{equation}
\label{functional_2}
\begin{aligned}
    C_2(\dot{\bm \Gamma} \bm y) 
    &= \left\|\frac{\partial v_{\bm r}}{\partial t}\right\|_{\mathcal{H}(\Omega_y(t))}^2\\ 
    &= \left\| Q_{\bm y}(v_{\bm r})\right\|_{\mathcal{H}(\Omega_y(t))}^2 \\
    &= \left\| \P_T(v_{\bm r})Q_{\bm y}(v_{\bm r},\dot{\bm y})  \right\|_{\mathcal{H}(\Omega_y(t))}^2
    +
    \left\| \P_N(v_{\bm r}) Q_{\bm y}(v_{\bm r},\dot{\bm y}) \right\|_{\mathcal{H}(\Omega_y(t))}^2.\\
\end{aligned}
\end{equation}
The linear coordinate flow minimizing \eqref{functional_2}
aims at ``undoing'' the action of the operator $G_{\bm y}$ in coordinates $\bm y$. 
More precisely the PDE solution $u(\bm x,t)$ as seen in coordinate $\bm y(\bm x,t)=\bm \Gamma(t)\bm x$, i.e., $v_{r}(\bm \Gamma(t)\bm x,t) \approx u(\bm x,t)$, 
varies in time as little as possible. 
In the third equality of \eqref{functional_2} 
we used the fact that $\P_T(v_{\bm r})$ 
and $\P_N(v_{\bm r})$ are orthogonal projections. 
Following similar steps used to prove 
Proposition \ref{prop:TDVP_linear} it is straightforward to 
show that the critical points $\bm \Sigma(t)$ 
of \eqref{functional_2} satisfy the 
linear system \eqref{TDVP_linear_system} with 
\begin{equation}
\bm A = \int_{\Omega_y(t)}  {\rm vec}(\bm c(\bm y)) {\rm vec}(\bm c(\bm y))^{\top} d \bm y, \qquad 
\bm b = \int_{\Omega_y(t)}  G_{\bm y}(v_{\bm r}(\bm y,t))
{\rm vec}(\bm c(\bm y)) d \bm y. 
\end{equation}
%
%
}
With the optimal coordinate flow generator $\bm \Sigma(t)$ available 
we can compute the matrix $\bm \Gamma(t)$ 
appearing in \eqref{linear_coordinate_flow} by solving 
the matrix differential equation 
\begin{equation}
\label{coordinate_flow_MDE}
\begin{cases}
\displaystyle \frac{d \bm \Gamma(t)}{d t} 
= \bm \Sigma(t) \bm \Gamma(t), \\
\bm \Gamma(0) = \bm I_{d \times d}. 
\end{cases}
\end{equation}
Correspondingly, the PDE \eqref{coordinate_flow_PDE2} can be written along the optimal rank-reducing linear coordinate flow \eqref{linear_coordinate_flow} 
as %
\begin{equation}
\label{coordinate_flow_PDE_linear}
\begin{cases}
\displaystyle \frac{\partial v_{\bm r}(\bm \Gamma(t)\bm x,t)}{\partial t} = 
G_{\bm y}(v_{\bm r}(\bm y,t),\bm y) 
- \left(\bm \Sigma(t)\bm y(\bm x,t)\right)  \cdot \nabla_{\bm y} v_{\bm r}(\bm y(\bm x,t),t), \\
v_{\bm r}(\bm x,0) = \mathfrak{T}_{\bm r}(u_0(\bm x)). 
\end{cases}
\end{equation}
Using the linear system \eqref{TDVP_linear_system}, the matrix differential equation \eqref{coordinate_flow_MDE}, and the 
PDE \eqref{coordinate_flow_PDE_linear} we can specialize 
the time integration scheme described in Section 
\ref{sec:TDVP} 
for nonlinear coordinate flows to linear coordinate flows. 
In Algorithm \ref{alg:TDVP} we provide pseudo-code for 
PDE time stepping with linear coordinate flows. 
In the Algorithm ${\rm TDVP}$ denotes 
a subroutine that solves the optimization problem 
\eqref{TDVP_linear} (e.g., by solving the symmetric 
$d^2 \times d^2$ linear 
system of equations \eqref{TDVP_linear_system}) and 
${\rm LMM}$ denotes any time stepping scheme. 


%
%
\begin{algorithm}[!t]
\SetAlgoLined
 \caption{PDE integrator with rank-reducing linear coordinate flow.}
\label{alg:TDVP}
\vspace{0.3cm}
 \KwIn{\vspace{0.1cm}\\\\
   $v_0$ $\rightarrow$ initial condition \\
   $\Delta t$ $\rightarrow$ temporal step size \\
   $N_t$ $\rightarrow$ total number of time steps 
 }
 \vspace{0.3cm}
 \KwOut{\vspace{0.1cm}\\
  $\bm \Gamma$ $\rightarrow$ linear coordinate transformation at final time $t_f$ \\
  $v(\bm x, t_f)= u(\bm \Gamma \bm x, t_f)$ $\rightarrow$ Solution at time $t_f$ on transformed coordinate system}
\vspace{0.3cm}
{\bf Runtime:} 
\begin{itemize}
\item [] $\bm \Gamma = \bm I_{d \times d}$
\item [] {\bf for $k = 0$ to $N_t$}
\begin{itemize}
 \item [] $\dot{\bm\Gamma}_k = {\rm TDVP}(v_k, G_{\bm y})$ \\
 \item [] $v_{k+1} = {\rm LMM}(v_{k}, G_{\bm y}, \Delta t) $ \\
 \item [] $\bm \Gamma_{k+1} = 
 		  {\rm LMM}(\bm \Gamma_{k},\dot{\bm \Gamma}_{k},\Delta t) $ \\
\end{itemize}
  {\bf end}
\end{itemize}
\end{algorithm}

\section{Numerical examples}
\label{sec:numerics}
We demonstrate the proposed coordinate-adaptive 
time integration scheme based on linear coordinate flows 
(Algorithm \ref{alg:TDVP}) on several prototype PDEs
projected on functional tensor train (FTT) 
manifolds $\mathcal{M}_{\bm r}$ \cite{Dektor_dyn_approx}. 
%
%
%

\subsection{Liouville equation} 
Consider the $d$-dimensional Liouville equation 
\begin{equation}
\label{advection_equation}
\begin{cases}
\displaystyle\frac{\partial u(\bm x,t)}{\partial t} = 
\bm f(\bm x) \cdot \nabla u(\bm x,t),\vs  \\
u(\bm x,0) = u_0(\bm x),
\end{cases}
\end{equation}
where $\bm f(\bm x)$ is a divergence-free vector field. 
We set the initial condition to be a product of 
independent Gaussians
\begin{equation}
\label{PDF_IC}
u_0(\bm x) = \frac{1}{m} \exp\left(- \sum_{j=1}^d 
\frac{1}{\beta_{j}} \left(x_j + t_{j} \right)^2 \right) , 
\end{equation}
where $\beta_j > 0$, $t_j\in \mathbb{R}$ and $m$ 
is the normalization factor
\begin{equation}
\nonumber
m = \left\|\exp\left(- \sum_{j=1}^d 
\frac{1}{\beta_{j}} \left(x_j + t_{j} \right)^2 \right)
\right\|_{L^1\left(\mathbb{R}^d\right)}.
\end{equation}
%
%
%
%

\subsubsection{Two-dimensional simulation results}
We first consider the Liouville on a two-dimensional 
flat torus $\Omega\subset \mathbb{R}^2$ with velocity vector
$\bm f(\bm x)=(f_1(\bm x),f_2(\bm x))$ consisting of 
\begin{equation}
\label{vrtx_vf}
f_1(\bm x) = \frac{\partial \psi}{\partial x_2}, \qquad
f_2(\bm x) = -\frac{\partial \psi}{\partial x_1}
\end{equation}
generated 
via the two-dimensional stream function \cite{Venturi_2012}
\begin{equation}
\psi(x_1,x_2) = \Theta(x_1) \Theta(x_2)
\end{equation}
with 
\begin{equation}
\label{ns_eig_fun}
\Theta(x) = \frac{\cos(\alpha x/L)}{\cos(\alpha/2)} - \frac{\cosh(\alpha x/L)}{\cosh(\alpha/2)},
\end{equation}
$L=30$ and  $\alpha = 4.73$.
This yields a two-dimensional 
measure-preserving (divergence-free) velocity 
vector $\bm f(\bm x)$.
In the initial condition \eqref{PDF_IC} we set 
$\bm \beta = \begin{bmatrix}
1/4 & 2
\end{bmatrix}$ and 
$\bm t = \begin{bmatrix}
3 & 3 
\end{bmatrix}$ 
resulting in a rank-$1$ initial condition.
\begin{figure}[!t]
\centerline{\footnotesize\hspace{-.0cm} (a) 
\hspace{4.6cm}  (b)  \hspace{4.8cm} (c) }
\centering
\includegraphics[scale=0.293]{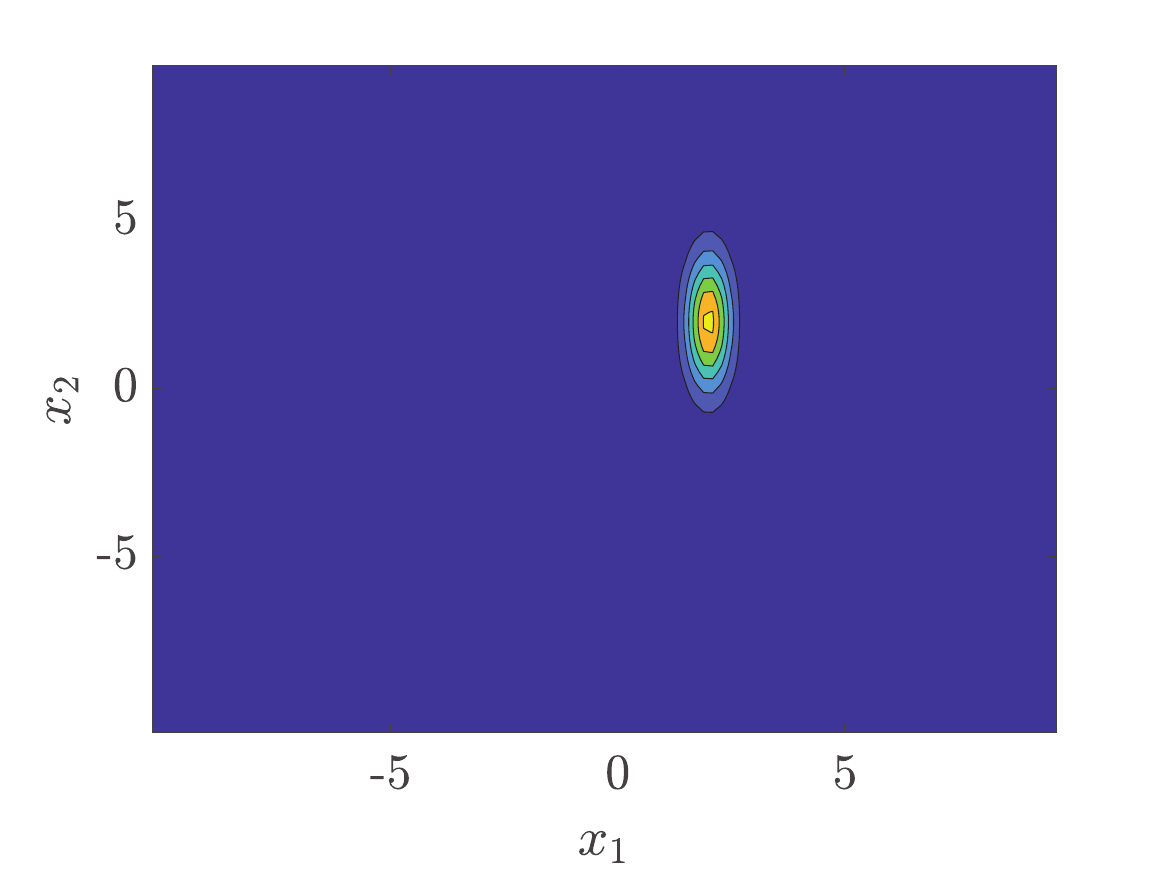} 
\hspace{-.6cm}
\includegraphics[scale=0.293]{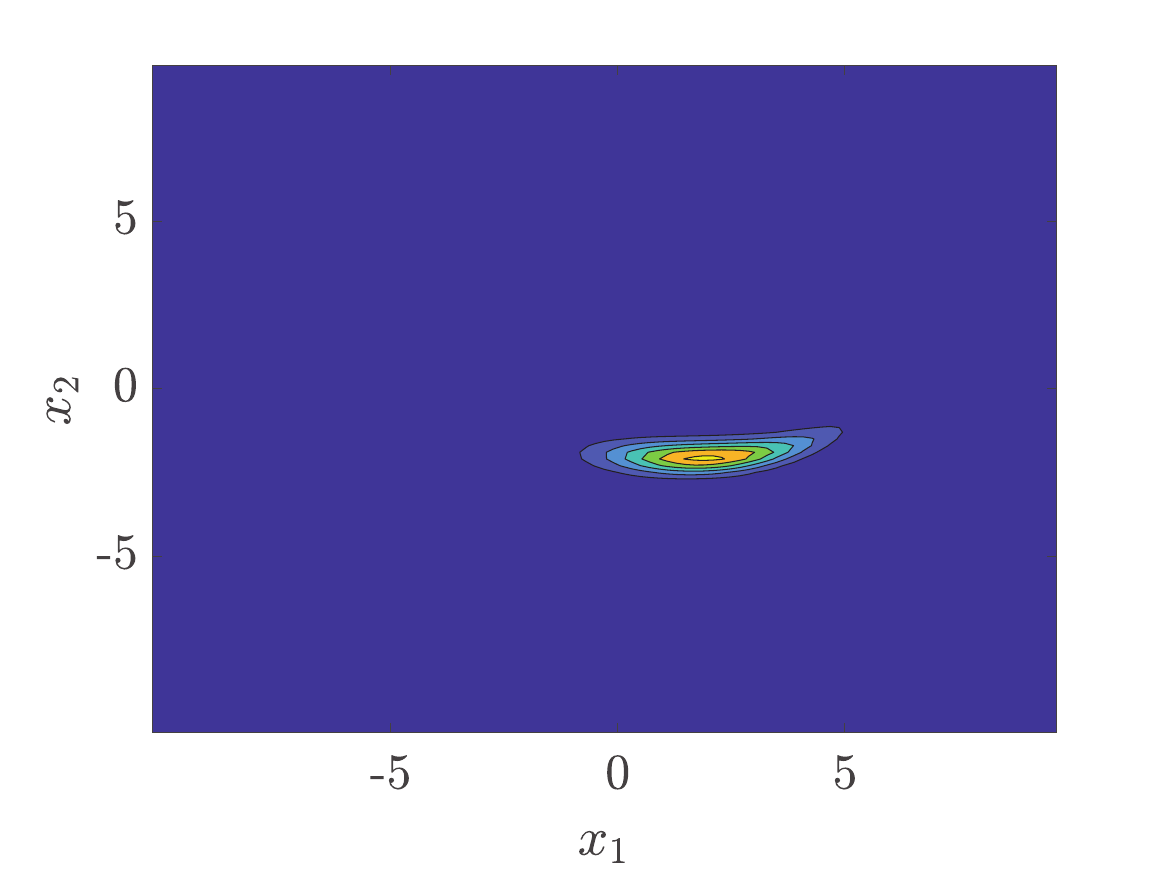} 
\hspace{-.6cm}
\includegraphics[scale=0.293]{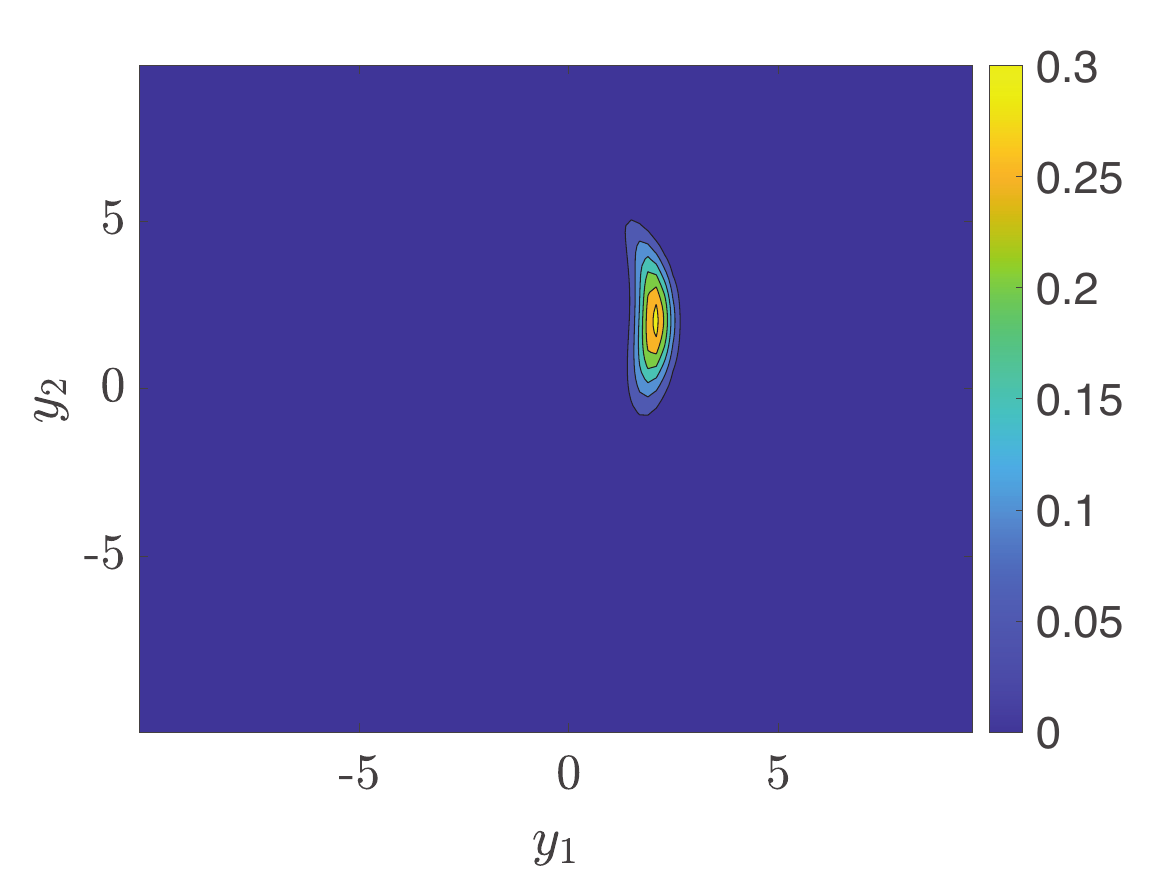} 
\caption{Two-dimensional advection PDE. (a) Initial condition. (b) Solution in Cartesian coordinates at time $t=35$. (c) Solution in time-dependent adaptive coordinates at time $t=35$.}
\label{fig:vortex_PDE_contours}
\end{figure}
Note that this is the same problem we recently studied 
in \cite{Coordinate_flows} using coordinate-adaptive 
tensor integration based on non-convex relaxations of 
the rank minimization problem. 
We discretize the PDE  with the Fourier pseudo-spectral 
method  \cite{spectral_methods_book} using $200$ points in each spatial variable $x_i$ and integrate the PDE 
from $t=0$ to $t=35$ using the two-steps 
Adams-Bashforth scheme (AB2) with time step 
size $\Delta t = 10^{-3}$. 
To demonstrate rank reduction with coordinate flows, we 
run three distinct simulations. In the 
first simulation we integrate the PDE on 
a full two-dimensional tensor product grid 
in Cartesian coordinates for all time 
resulting in our benchmark solution. 
In the second simulation we construct a low-rank 
representation of the solution in fixed Cartesian coordinates 
with the rank-adaptive step-truncation time integration scheme \cite{rodgers2020step-truncation} built 
upon the AB2 integrator using relative tensor 
truncation tolerance $\delta = 10^{-8}$. 
The third simulation demonstrates the low-rank 
tensor format in adaptive coordinates generated 
by Algorithm \ref{alg:TDVP} (using cost 
function \ref{functional_2}) with a 
step-truncation time integration scheme built 
upon the AB2 integrator using relative 
tensor truncation tolerance $\delta = 10^{-8}$.

In Figure \ref{fig:vortex_PDE_contours} we plot 
the low-rank solution at time $t=0$ (a), $t=35$ 
computed in fixed Cartesian coordinates (b) 
and $t=35$ computed in adaptive coordinates (c). 
We notice that the time-dependent coordinate 
system in the adaptive simulation captures 
the affine components (i.e., translation, 
rotation, and stretching) 
of the dynamics generated by the PDE operator. 
In Figure \ref{fig:vortex_PDE_ranks_normals}
(d) we plot the rank versus time of the solution 
computed in fixed Cartesian coordinates 
and the solution computed in adaptive 
time-dependent coordinates. 
We also plot the norm of the normal components, 
i.e., the left hand side of \eqref{normal_bound}, 
versus time in Cartesian coordinates and in adaptive coordinates. 
Since the linear coordinate flow is minimizing the normal 
component of the PDE operator we observe that 
the red dashed line increases at a slower rate 
than the blue dashed line. 
Once the norm of the normal component reaches a threshold 
depending on the relative tensor truncation tolerance $\delta$ set on the solution and time step size $\Delta t$, the rank-adaptive criterion discussed \cite{adaptive_rank} is triggered which results in an increase of the solution rank. 
At the subsequent time step the normal component 
is reduced due to the addition of a tensor mode. 
%
%
\begin{figure}[!t]
\centerline{\footnotesize\hspace{.5cm} (a) 
\hspace{7.8cm}  (b)   }
\centering
\includegraphics[scale=0.4]{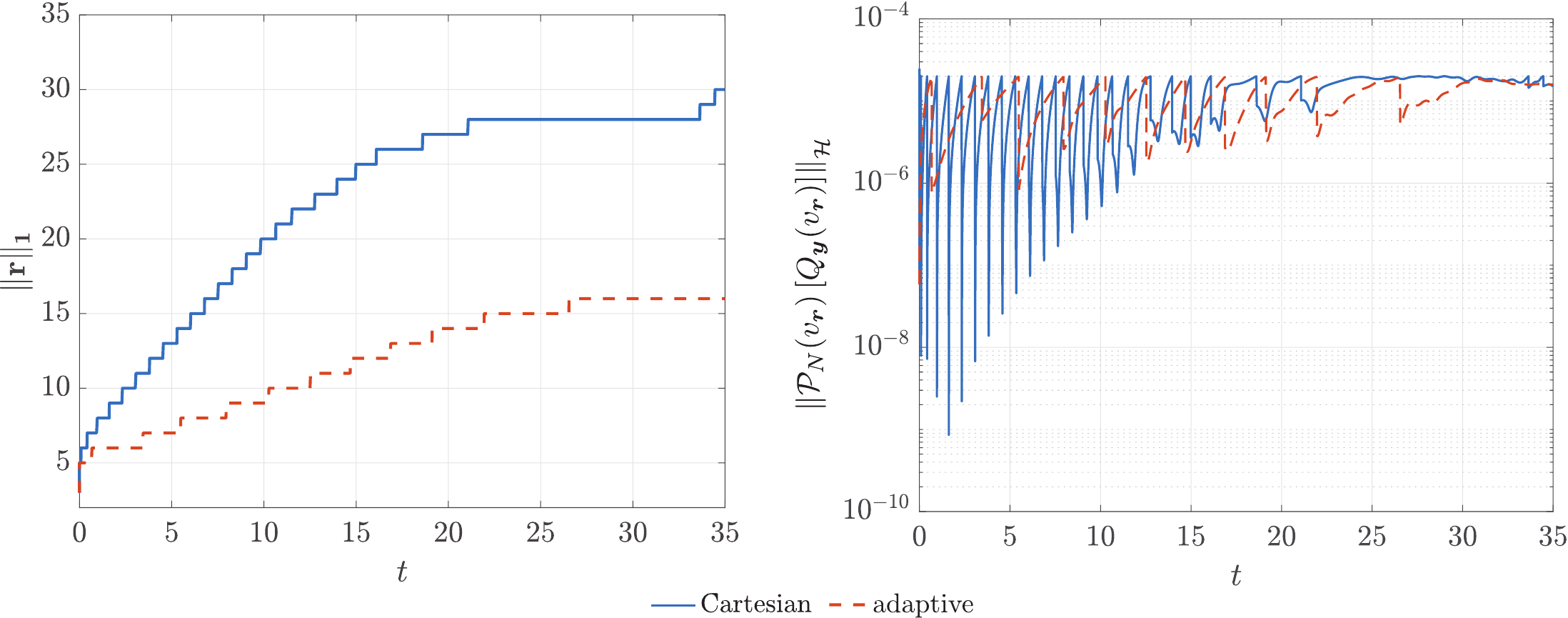} 
\caption{Two-dimensional advection PDE. (a) Rank versus time and (b) norm of the normal component 
of PDE dynamics versus time for the 
low-rank solution in Cartesian coordinates (blue solid line) and in adaptive coordiantes (red dashed line).}
\label{fig:vortex_PDE_ranks_normals}
\end{figure}
At time $t = 35$ we check the $L^{\infty}$ 
error of the low-rank tensor solution computed 
in Cartesian coordinates and the low-rank 
solution computed using adaptive coordinates 
(for the adaptive solution we map it back to Cartesian 
using a two-dimensional interpolant) and 
find that both errors are bounded by $1.5 \times 10^{-5}$.

\subsubsection{Three-dimensional simulation results}
\label{sec:3D_advection_results}
Next we consider the Liouville equation  
\eqref{advection_equation} on a three-dimensional 
flat torus $\Omega\subset \mathbb{R}^3$ with velocity vector 
\begin{equation}
\label{drift_3D}
\bm f(\bm x) = 
\begin{bmatrix}
x_2 \\
-x_1 \\
x_2 \\
\end{bmatrix}.
\end{equation}
This allows us to test our 
coordinate-adaptive algorithm for 
a problem in which we know the optimal ridge matrix. 
We set the parameters of the initial condition \eqref{PDF_IC} as 
$\bm \beta = \begin{bmatrix}
4 & 1/4 & 4
\end{bmatrix}$ and 
$\bm t = \begin{bmatrix}
1 & -1 & 1 
\end{bmatrix}$
resulting in the FTT rank 
$\bm r(0)=\begin{bmatrix}
1 & 1 & 1 & 1
\end{bmatrix}$. 
Due to the choice of coefficients 
\eqref{drift_3D}, the 
analytical solution to the PDE 
\eqref{advection_equation} 
can be written as a ridge 
function in terms of the PDE 
initial condition 
\begin{equation}
\label{3D_tensor_ridge_soln}
u(\bm x,t) = u_0\left(e^{t\bm B} \bm x \right), 
\end{equation}
where
\begin{equation}
\bm B = 
\begin{bmatrix}
0 & 1 & 0 \\
-1 & 0 & 0 \\
0 & 1 & 0
\end{bmatrix}.
\end{equation}
Thus \eqref{3D_tensor_ridge_soln} is a tensor 
ridge solution to the 3D Lioville equation 
\eqref{advection_equation} with the same 
rank as the initial condition, i.e., 
in this case there exists a tensor ridge solution 
at each time with FTT rank equal to 
$\begin{bmatrix}
1 & 1 & 1 & 1
\end{bmatrix}$. 

To demonstrate the performance of the rank-reducing coordinate 
flow algorithm, we ran two low-rank simulations of the PDE \eqref{advection_equation} up to $t = 1$ 
and compared the results with the analytic solution \eqref{3D_tensor_ridge_soln}. 
The first simulation is computed using the 
rank-adaptive step-truncation methods discussed in
\cite{adaptive_rank,rodgers2020step-truncation} with 
relative truncation accuracy $\delta = 10^{-5}$, Cartesian coordinates, and time step size $\Delta t = 10^{-3}$. 
The second simulation demonstrates the performance of the 
rank-reducing coordinate flow (Algorithm \ref{alg:TDVP}) and is computed with the same step-truncation scheme as the simulation in Cartesian coordinates. 
For this example the coordinate flow is constructed by minimizing 
the functional \eqref{functional_2}. 

In Figure \ref{fig:3D_PDE_rank_and_error}(a) 
we plot the $1$-norm of the 
FTT solution rank versus time. 
We observe that the FTT solution rank in Cartesian 
coordinates grows rapidly while the FTT-ridge 
solution generated by the proposed Algorithm 
\ref{alg:TDVP} recovers the optimal time-dependent low-rank tensor ridge solution, which remains constant rank for all time at $\left\|\bm r(t)\right\|_{1} = 4$. 
We also map the FTT ridge solution back to Cartesian coordinates using a three-dimensional trigonometric 
interpolant every $100$ time steps and compare both 
low-rank FTT solutions to the benchmark solution. 
In Figure \ref{fig:3D_PDE_rank_and_error}(b) we plot 
the $L^{\infty}$ error versus time. 
The solution computed in adaptive coordinates does not change in time and therefore is about one order of magnitude more accurate than the solution computed in Cartesian coordinates. 
In this case the proposed algorithm outperforms 
our previous algorithm proposed in \cite{Coordinate_flows} which does not recover 
the optimal linear coordinate transformation 
due to the non-convexity of the cost function used 
to generate the time-dependent coordinate system. 
\begin{figure}[!t]
\centerline{\footnotesize\hspace{.5cm} (a) 
\hspace{7.3cm}  (b)   }
\centering
\includegraphics[scale=0.4]{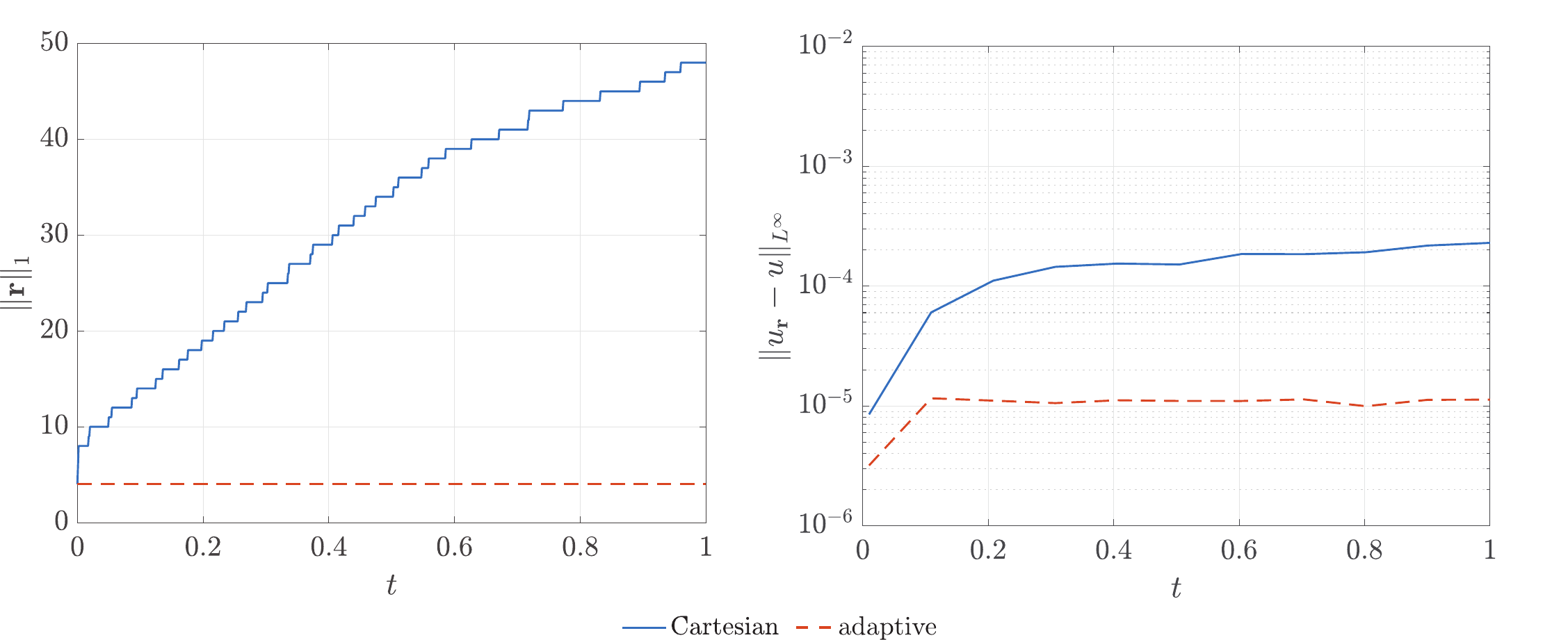} 
\caption{Three-dimensional advection equation \eqref{advection_equation}. (a) Rank versus time and (b) $L^{\infty}$ error of the 
low-rank solutions relative to the 
benchmark solution.}
\label{fig:3D_PDE_rank_and_error}
\end{figure}
For this example the PDE operator has separation rank (CP-rank) $3$ in Cartesian coordinates and CP-rank $9$ in a general linearly transformed coordinate system. 
%
%

\subsection{Fokker-Planck equation}
Finally we consider the Fokker-Planck equation 
\begin{equation}
\label{fp-equation}
\begin{cases}
\displaystyle\frac{\partial u(\bm x,t)}{\partial t} = 
- \nabla \cdot \left( \bm f(\bm x) u(\bm x,t)\right)  + 
\sum_{i,j=1}^d \frac{\partial^2}{\partial x_i \partial x_j} \left( D_{ij}(\bm x) u(\bm x,t)\right),\vs  \\
u(\bm x,0) = u_0(\bm x),
\end{cases}
\end{equation}
in three spatial dimensions ($d=3$) on a flat torus 
$\Omega\subseteq \mathbb{R}^3$. 
We set the drift vector $\bm f$ as in \eqref{drift_3D}, 
and diffusion tensor 
\begin{equation}
    \bm D = \sigma \begin{bmatrix}
        \exp\left(-x_2^2\right) & 0 & 0 \\
        0 & \exp\left(-x_3^2\right)& 0 \\
        0 & 0 & \exp\left(-x_1^2\right)
        \end{bmatrix}
\end{equation}
with $\sigma = 1/4$. 
We use the same initial condition as in Section \ref{sec:3D_advection_results}.

Once again, to demonstrate the performance of the rank-reducing coordinate flow algorithm, we ran three simulations of the PDE \eqref{fp-equation} up to $t = 1$. 
The first simulation is computed with a 
three-dimensional Fourier pseudo-spectral 
method on a tensor product grid with 200 points per dimension 
and AB2 time stepping with $\Delta t = 10^{-4}$. 
This yields an accurate solution benchmark.  
The second simulation is computed using the 
rank-adaptive step-truncation methods discussed in
\cite{adaptive_rank,rodgers2020step-truncation} using relative truncation accuracy $\delta = 10^{-5}$, fixed Cartesian coordinates, 
and time step size $\Delta t = 10^{-3}$. 
The third simulation demonstrates the performance of the rank-reducing coordinate flow (Algorithm \ref{alg:TDVP}) and is computed with the same step-truncation scheme as the simulation in Cartesian coordinates. 
For this example the coordinate flow is constructed by minimizing the functional \eqref{functional_2}. 

\begin{figure}[!t]
\centerline{\footnotesize $t=0.1$ \hspace{3.95cm} $t=0.5$
\hspace{3.75cm} $t = 1.0$  }
\centering
\rotatebox{90}{\hspace{.6cm}\footnotesize Cartesian coordinates}
\hspace{0.1cm}
\includegraphics[scale=0.25]{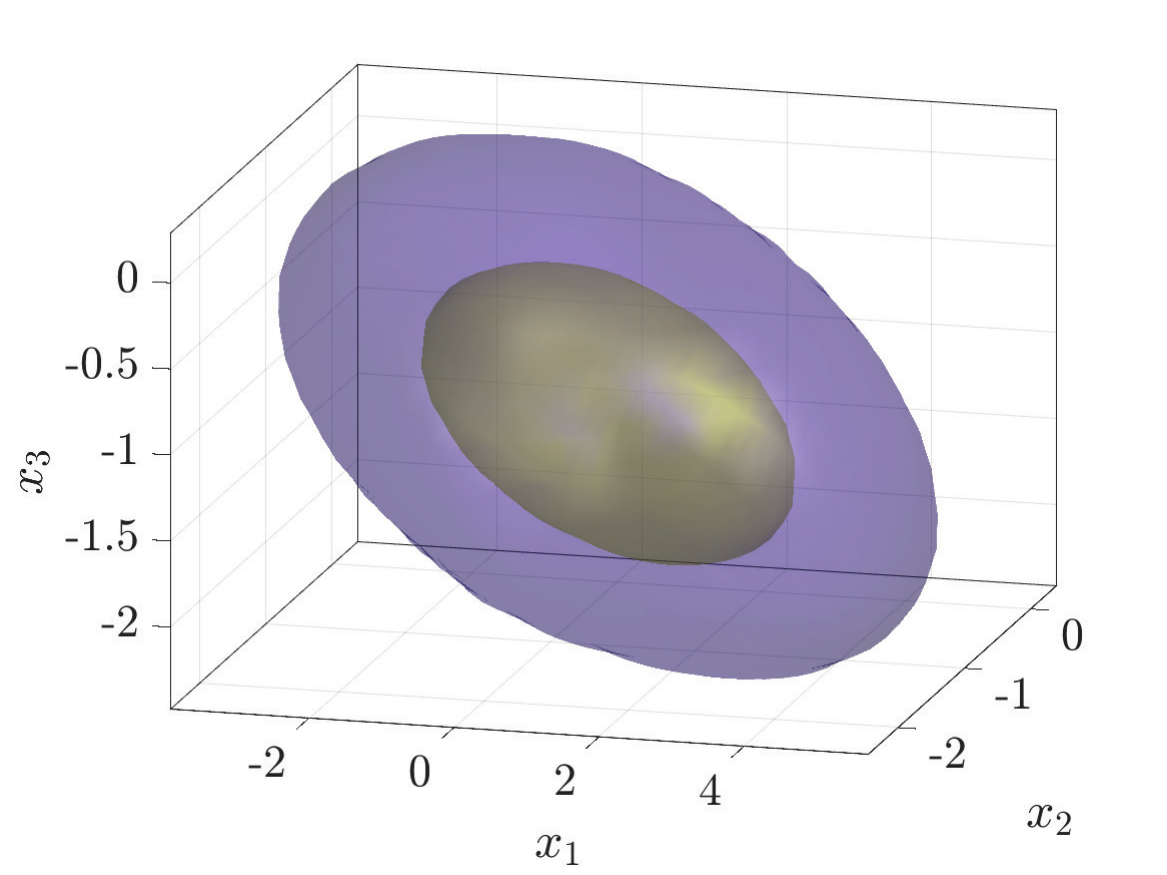} 
\includegraphics[scale=0.25]{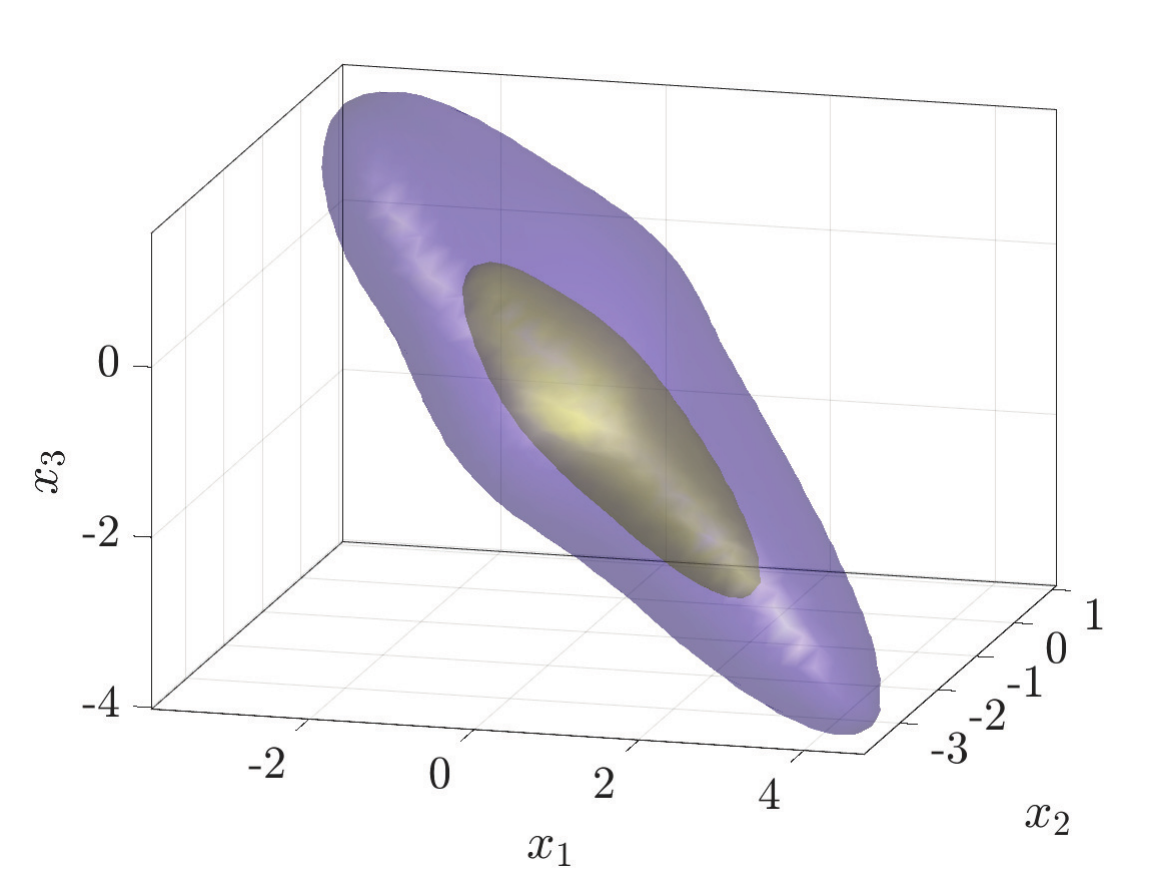} 
\includegraphics[scale=0.25]{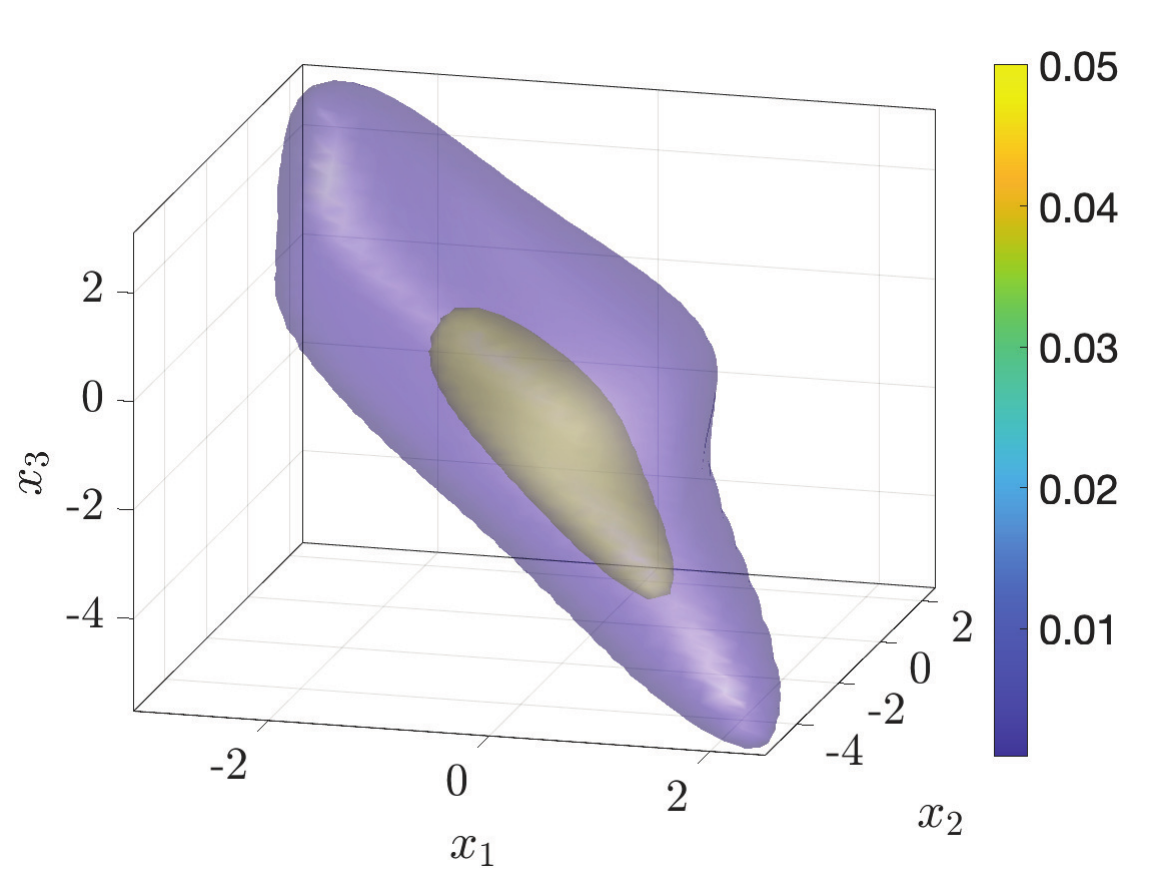} \\
\rotatebox{90}{\hspace{.6cm}\footnotesize adaptive coordinates}
\hspace{0.1cm}
\includegraphics[scale=0.25]{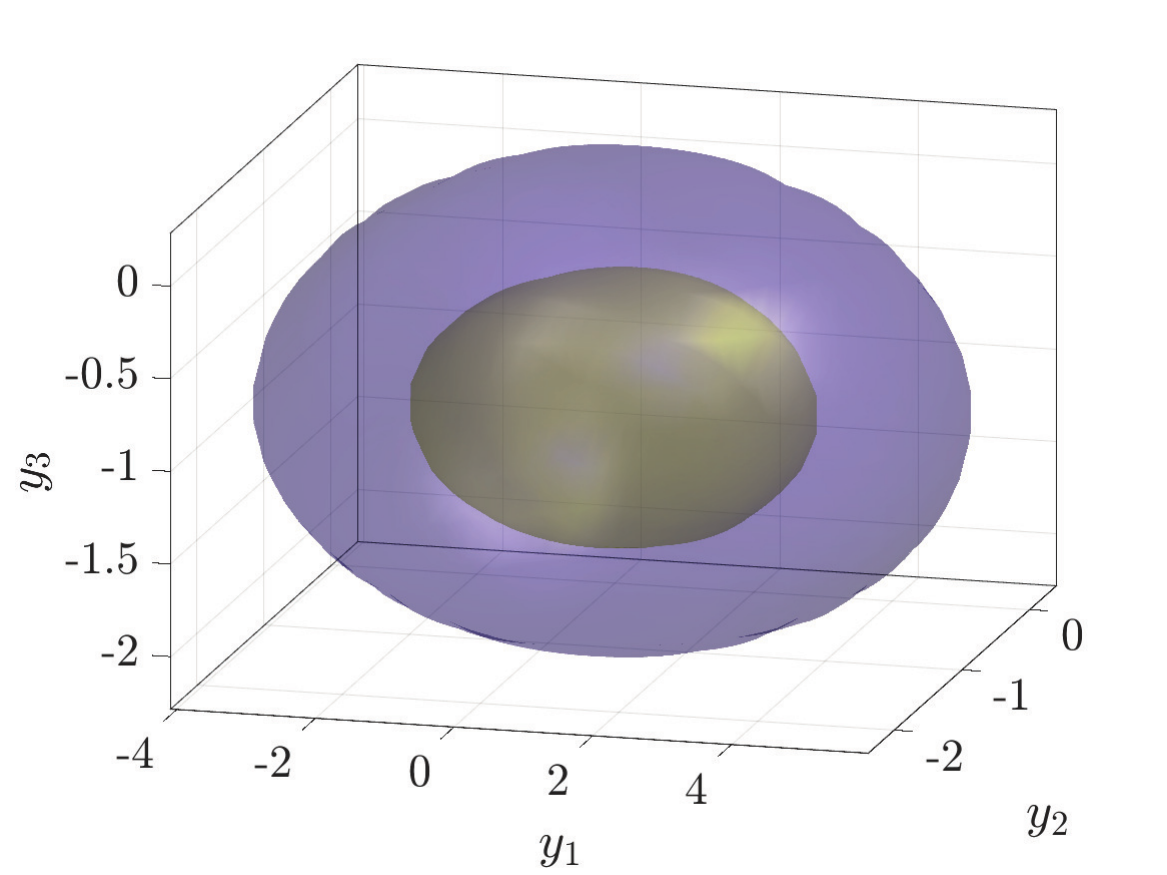} 
\includegraphics[scale=0.25]{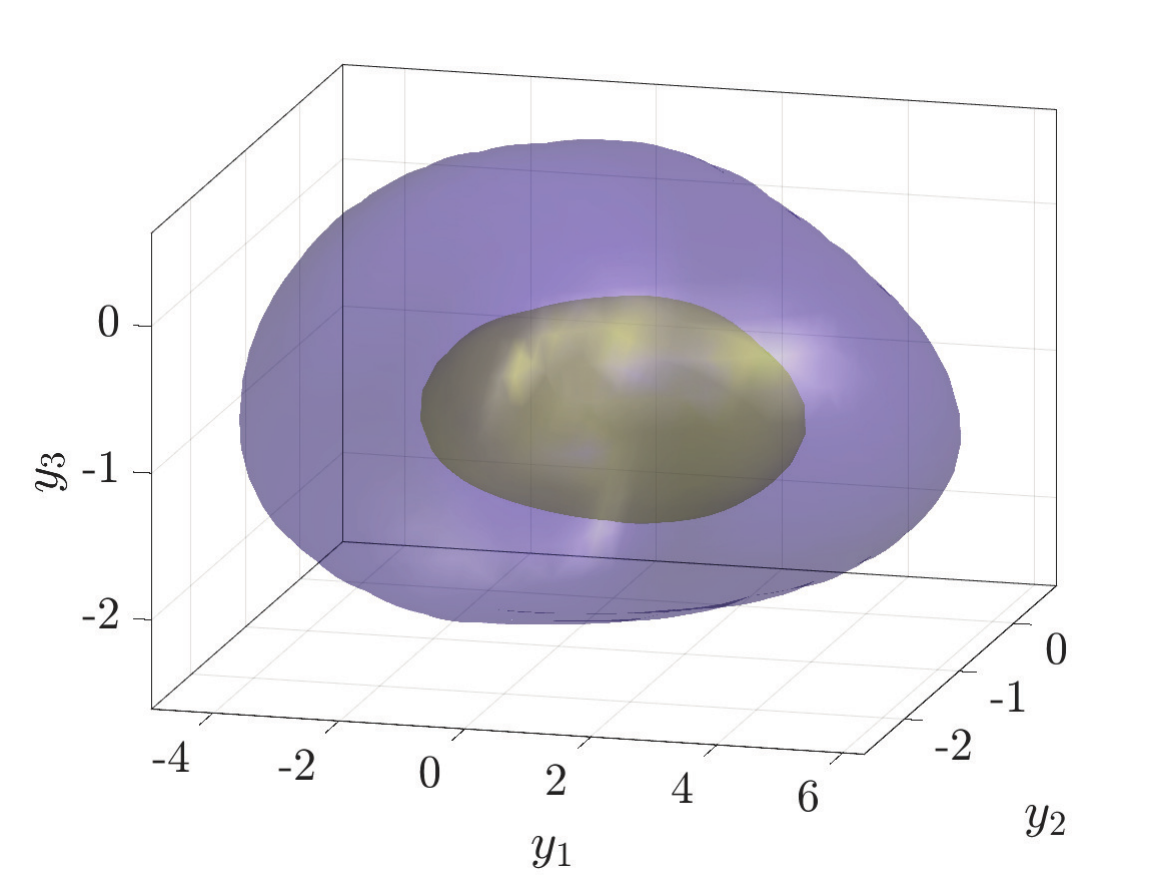} 
\includegraphics[scale=0.25]{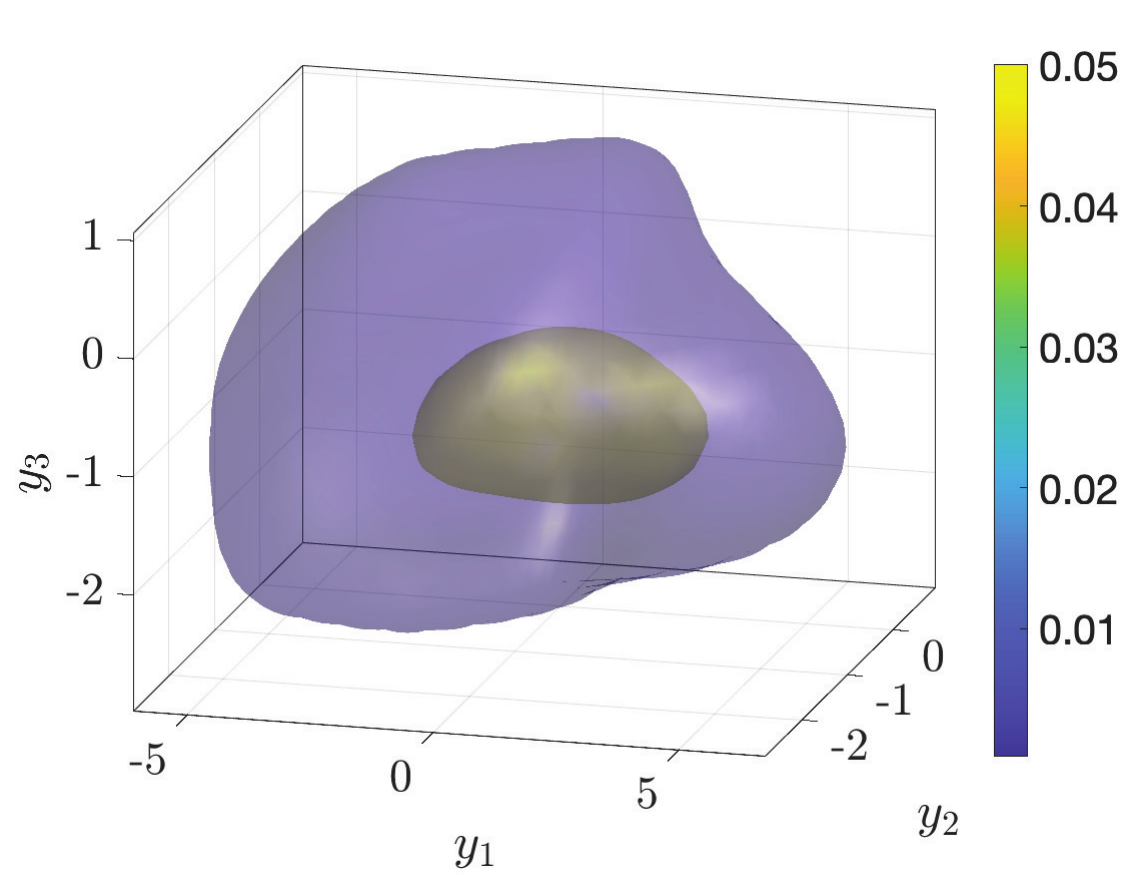} 
\caption{Three-dimensional Fokker-Planck equation \eqref{fp-equation}. 
Time snapshots of the solution at $t=0.1,0.5,1.0$ in Cartesian coordinates (top row) and adaptive coordinates (bottom row). The adaptive coordinate system generated by Algorithm \ref{alg:TDVP} effectively captures the affine effects (i.e., rotation, stretching, translation) of the transformation generated by the Fokker-Planck equation \eqref{fp-equation}.}
\label{fig:fp_snapshots}
\end{figure}
In Figure \ref{fig:fp_snapshots} we plot three 
time snapshots of the low-rank solutions 
to \eqref{fp-equation} in fixed Cartesian coordinates 
(top row) and rank-reducing adaptive coordinates (bottom row). 
We observe that the rank-reducing adaptive coordinate transformation 
effectively captures the affine effects of the transformation generated by the PDE \eqref{fp-equation} even in the presence of diffusion.
In Figure \ref{fig:fp_rank_and_error}(a) 
we plot the $1$-norm of the 
FTT solution rank versus time. 
We observe that the FTT solution rank in Cartesian 
coordinates grows rapidly while the FTT-ridge 
solution generated by Algorithm 
\ref{alg:TDVP} grows significantly slower. 
We also map the FTT ridge solution back to Cartesian coordinates using a three-dimensional trigonometric interpolant every 
$100$ time steps and compare both 
low-rank FTT solutions to the benchmark solution. 
In Figure \ref{fig:fp_rank_and_error}(b) we plot the $L^{\infty}$ error of Cartesian and coordinate-adaptive low-rank solutions relative to the benchmark solution versus time. 
We observe that the process of solving the PDE 
in transformed coordinates and mapping the transformed solution back to Cartesian coordinates incurs negligible error. 
%
%
\begin{figure}[!t]
\centerline{\footnotesize\hspace{.5cm} (a) 
\hspace{7.3cm}  (b)   }
\centering
\includegraphics[scale=0.4]{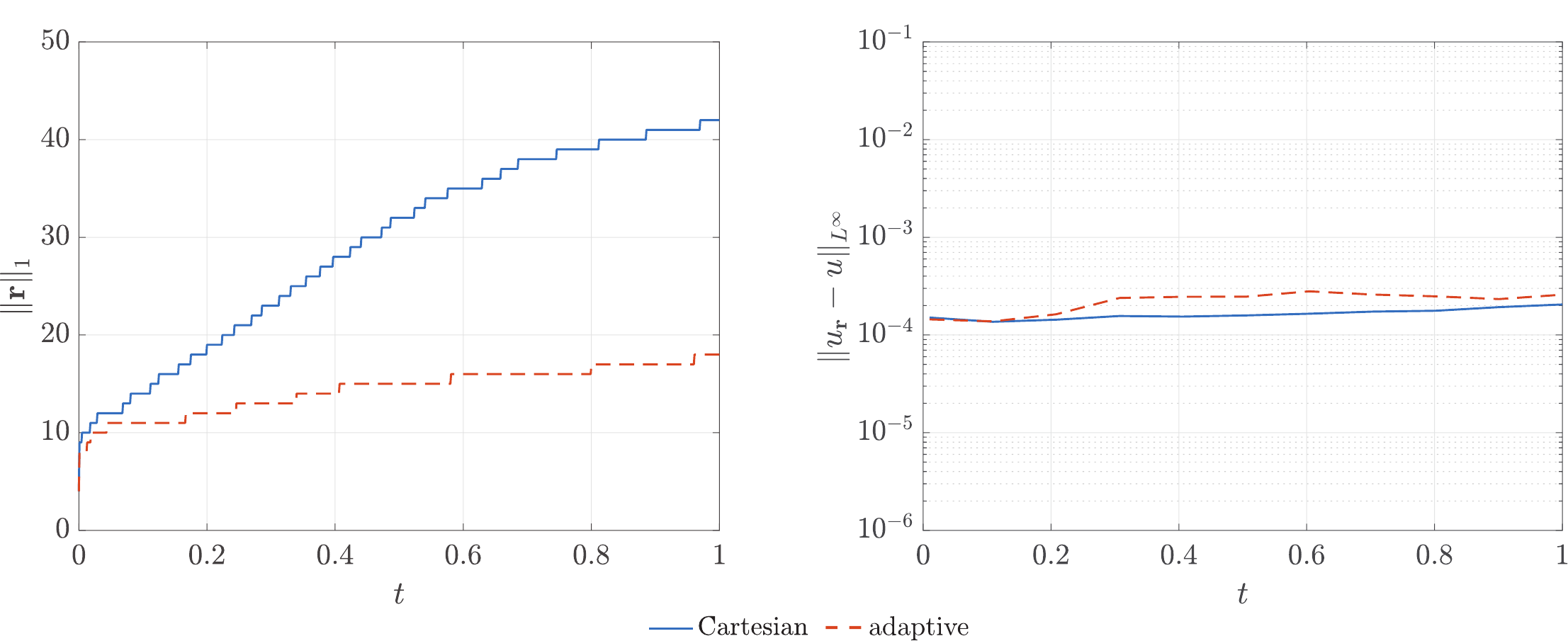} 
\caption{Three-dimensional Fokker-Planck equation \eqref{fp-equation}. 
(a) Rank versus time and (b) 
$L^{\infty}$ error of the 
low-rank solutions relative to the 
benchmark solution. }
\label{fig:fp_rank_and_error}
\end{figure}
%

\section{Conclusions}
\label{sec:conclusions}
We proposed a new tensor integration method for time-dependent 
PDEs that controls the rank of the PDE solution in time by using  
diffeomorphic coordinate transformations. Such coordinate 
transformations are generated by minimizing the normal 
component of the PDE operator (written in intrinsic coordinates) 
relative to the tensor manifold that approximates the PDE solution. 
This minimization principle is defined by a convex functional which can be computed efficiently and has optimality guarantees. 
The proposed method significantly improves upon and may be used in conjunction with the coordinate-adaptive algorithms we recently proposed 
in \cite{Coordinate_flows}, which are based on non-convex relaxations of the rank minimization
problem and Riemannian optimization.  
We demonstrated the proposed coordinate-adaptive time-integration algorithm for linear coordinate transformations on prototype Liouville and Fokker-Planck equations in two and three dimensions. 
Our numerical results clearly demonstrate 
that linear coordinate flows can capture the 
affine component (i.e., rotation, translation, and stretching) of the transformation generated by PDE operator very effectively. 
In general, one cannot expect linear coordinate flow (or even nonlinear coordinate flows) to fully control the rank of the solution generated by an arbitrary nonlinear PDE.
Yet, the proposed method allows us to solve certain 
classes of PDEs at a computational cost that is 
significantly lower than standard temporal integration 
on tensor manifolds in Cartesian coordinates. 
In general overall computational cost of the proposed method is not only determined by the tensor rank of the solution but also the rank of the PDE operator $G_{\bm y}$, as discussed in 
\cite{Coordinate_flows}. 
Further research is warranted to determine if the 
efficiency of the proposed coordinate-adaptive methodology can be improved by including conditions on the tangent projection of $\dot{\bm y} \cdot \nabla v_{\bm r}$ in \eqref{TDVP}, 
or by simultaneously controlling the PDE solution and operator rank during time integration.

\section*{Acknowledgements}
\noindent
This research was supported by the U.S. Air Force 
Office of Scientific Research grant FA9550-20-1-0174, and by the U.S. Army Research Office 
grant W911NF1810309.

\bibliographystyle{plain}
\bibliography{bibliography_file}

\end{document}